\documentclass{amsart}

\usepackage{amssymb,amsmath,amsfonts}

\newtheorem{theorem}{Theorem}[section]
\newtheorem{lemma}[theorem]{Lemma}

\theoremstyle{definition}

\newtheorem{corollary}[theorem]{Corollary}
\theoremstyle{remark}

\numberwithin{equation}{section}
\DeclareMathOperator{\ind}{ind}
\DeclareMathOperator{\rank}{rank}

\begin{document}

\title{Representations for the Drazin inverse of the generalized Schur complement}

\author{Daochang Zhang, Xiankun Du}
\address{School of Mathematics, Jilin University, Changchun 130012, China }
\email{daochangzhang@gmail.com}
\email{duxk@jlu.edu.cn}
\thanks{This work was supported by NSFC(Grant No. 11371165)}

\date{December 4, 2013.}


\begin{abstract}
 In this paper we present expressions for the Drazin inverse of the generalized Schur complement $A-CD^{d}B$ in terms of the Drazin inverses of $A$ and the generalized Schur complement $D-BA^{d}C$ under less and weaker restrictions, which generalize several results in the literature and the formula of Sherman-Morrison-Woodbury type.
\end{abstract}

\maketitle

\section{introduction}
Let  $\mathbb{C}^{n\times n}$ denote the set of $n\times n$ complex matrices. For a matrix $A \in \mathbb{C}^{n\times n}$, the Drazin inverse of $A$ is the unique matrix $A^d$ such that
\begin{equation*}\label{defin1}
  AA^d=A^dA,~~~ A^dAA^d=A^d,~~~ A^{k}=A^{k+1}A^d,
\end{equation*}
 where $k$ is the smallest non-negative integer such that $\rank(A^{k}) = \rank(A^{k+1})$, called index of $A$ and denoted by $\ind(A)$. The Drazin inverse is a generalization of inverses and group inverses of matrices, which was introduced by M. P. Drazin in \cite{Drazin(1958)} for associative rings and semigroups. There is widespread applications of  Drazin inverses of complex matrices in various fields,
such as differential equations, control theory, Markov chains, iterative methods and so on
(see \cite{Israel (2003),Campbell (1991)}).

In 2002, Wei \cite{Wei (2002)} derived explicit expressions of the Drazin inverse of a modified  matrix $A-CB$ under certain circumstances, which extend results of \cite{Meyer (1980),Shoaf (1975)} and can be applied to update finite Markov chains.

In 2012, Cvetkovi\'{c}-Ili\'{c} and Ljubisavljevi\'{c} in \cite{Cvetkovic (2012)} generalize results of \cite{Wei (2002)}
to the generalized Schur complement $A-CD^{d}B$ under the following conditions:
$A^{\pi}C=0,~BA^{\pi}=0,~CD^{\pi}Z^{d}B=0,~CD^{d}Z^{\pi}B=0, ~CZ^{d}D^{\pi}B=0 ~\text{and}~CZ^{\pi}D^{d}B=0$,
where $Z$ denotes the generalized Schur complement $D-BA^dC$ of $A$. This result also extends the formula of Sherman-Morrison-Woodbury type
$$
  (A-CD^{-1}B)^{-1}=A^{-1}+A^{-1}C(D-BA^{-1}C)^{-1}BA^{-1},
$$
where the matrices $A,D$ and the Schur complement $D-BA^{-1}C$ are invertible. Similar generalizations are obtained by Dopazo and Mart\'{i}nez-Serrano \cite{Dopazo (2013)} in 2013
 under the following conditions:
$$
A^{\pi}C=0, ~CD^{\pi}=0, ~D^{\pi}B=0, ~CZ^{\pi}=0~\text{and}~ Z^{\pi}B=0.
$$
Recently, Dijana \cite{Dijana 2013} and Shakoor, Yang, Ali \cite{Shakoor} gave representations for the Drazin inverses of generalized Schur complement $A-CD^{d}B$ under new conditions to generalize some results in the literature.

The aim of this paper is to present expressions for the Drazin inverse of the generalized Schur complement $A-CD^{d}B$ in terms of the Drazin inverses of $A$ and the generalized Schur complement $D-BA^{d}C$. In Section 2, by using Dedekind finiteness of unital subalgebras of $\mathbb{C}^{n\times n}$, we find that some assumptions of theorems in the literature are superfluous. We give representations of $(A-CD^{d}B)^d$ under less and weaker conditions, which generalizes results of
\cite{Cvetkovic (2012),Dijana 2013,Dopazo (2013),Shakoor,Wei (2002)}, of course, including the formula of Sherman-Morrison-Woodbury type.  In Section 3, we derive a new formula for the Drazin inverse of $A-CD^{d}B$ under the following conditions $A^{\pi}CD^{d}B=0,~D^{\pi}BA^{d}C=0$ and $D^{d}BAA^{d}=D^{d}DBA^{d}$, a corollary of which recovers a generalization of Jacobson's Lemma (see \cite[Theorem 3.6]{Castro (2010)}) for the case of matrices.

Throughout this paper, let $A$ and $D$ be square matrices (not necessarily with the same orders), $B$ and $C$ be matrices with suitable orders, and $A^{\pi}=I-AA^{d}$. The generalized Schur complements will be denoted by
\begin{equation*}\label{sign}
 S=A-CD^{d}B~~\text{and}~~Z=D-BA^{d}C.
 \end{equation*}

We adopt the convention that $A^0=I$, the identity matrix, for any square matrix $A$ and $\sum_{i=0}^k\ast=0$ in case $k<0$.
\section{Drazin inverse of the generalized Schur complement}
\begin{lemma}\label{P+Q}(\cite[Theorem 2.1]{Hartwig (2001)})
Let $P$ and $Q$ be ${n\times n}$ matrices. If $PQ=0$, then
$$
  (P+Q)^{d}=Q^{\pi}\sum_{i=0}^{t-1}Q^{i}(P^{d})^{i+1}+\sum_{i=0}^{s-1}(Q^{d})^{i+1}P^{i}P^{\pi},
$$
where $s= \ind(P)$ and $t= \ind(Q)$.
\end{lemma}

\begin{lemma}\label{(lemma) P nilpotent}
If $A^{\pi}CD^{d}B=0$, then
$$
  S^{d}=S_{A}^{d}+\sum_{i=0}^{k-1}(S_{A}^{d})^{i+2}SA^{i}A^{\pi},
$$
where
$S_{A}=SAA^{d}$ and $k= \ind(A)$.
\end{lemma}
\begin{proof} We first note that $S=SA^{\pi}+S_{A}$.
Since $A^{\pi}CD^{d}B=0$, we have $A^{\pi}S_A=0$ and $$(SA^\pi)^i=S(A^\pi S)^{i-1}A^\pi=S(AA^\pi)^{i-1}A^\pi=SA^{i-1}A^\pi$$ for any positive integer $i$. Let $k= \ind(A)$. It has been known that $k$ is the least nonnegative integer such that $A^kA^\pi=0$. Thus $SA^\pi$ is nilpotent and $k\leq\ind (SA^\pi)\leq k+1$, and so $(SA^{\pi})^{d}=0$ and $(SA^{\pi})^{\pi}=I$. Let $s= \ind(SA^{\pi})$. Then Lemma \ref{P+Q} implies that
$$
 S^{d}=\sum_{i=0}^{s-1}(S_{A}^{d})^{i+1}(SA^{\pi})^{i}=\sum_{i=0}^{s}(S_{A}^{d})^{i+1}(SA^{\pi})^{i}
      =S_{A}^{d}+\sum_{i=0}^{s-1}(S_{A}^{d})^{i+2}SA^{i}A^{\pi}.
$$
Since $s-1\leq k\leq s$ and $A^iA^\pi=0$ for any $i\geq k$, we have $$
 S^{d}= S_{A}^{d}+\sum_{i=0}^{k-1}(S_{A}^{d})^{i+2}SA^{i}A^{\pi},
$$
as desired.
\end{proof}

\begin{lemma}\label{ECE}
  Let $X,~Y$ and $e$ be $n\times n$ matrices and $e^{2}=e$. If $XeY=e$, then $eYeXe=e$.
\end{lemma}
\begin{proof}
  Let $W=\{M \in \mathbb{C}^{n\times n}~~|~~eM=Me=M\}$.
  Then $W$ is a finite dimensional algebra over $\mathbb{C}$ with identity $e$, and so $W$ is Dedekind finite (see \cite[Corollary 21.27]{Lam}).
  Note that $eXe,~eYe \in W$ and $(eXe)(eYe)=e$. Then $(eYe)(eXe)=e$, that is, $eYeXe=e$.
\end{proof}

If $\ind(A)=1$, then $A^{d}$ is called the group inverse of $A$ and denoted by $A^{\#}$.
In what follows, let $H=BA^{d}$ and $K=A^{d}C$.
\begin{lemma}\label{(lemma) Qd}
 Let $S_{A}=AA^{d}SAA^{d}$ and let $M=A^{d}+KZ^{d}H$. Then the following statements are equivalent:
 \begin{enumerate}
  \item $KD^{\pi}Z^{d}H=KD^{d}Z^{\pi}H$;
   \item $S_{A}M=AA^{d}$;
   \item $MS_{A}=AA^{d}$;
   \item $KZ^{\pi}D^{d}H=KZ^{d}D^{\pi}H$.
\end{enumerate}
Furthermore, under the conditions above $S_{A}$ has the group inverse
$$
  S_{A}^{\#}=A^{d}+KZ^{d}H.
$$
\end{lemma}
\begin{proof}
Let $A^{e}=AA^{d}$ and $Z^{e}=ZZ^{d}$. Then
$$
S_{A}M=A^{e}+AKZ^{d}H-AKD^{d}H-AKD^{d}(D-Z)Z^{d}H.
$$
Thus (2) holds if and only if $AK(Z^{d}-D^{d}-D^{d}(D-Z)Z^{d})H=0$,
or equivalently $K(D^{\pi}Z^{d}-D^{d}Z^{\pi})H=0$, that is,  (1) holds.
Similarly, (3) is equivalent to (4). Lemma~\ref{ECE} implies equivalence of (2) and (3). Furthermore, (2) and (3) give $
  S_{A}^{\#}=M=A^{d}+KZ^{d}H
$.
\end{proof}

Now we can extend the Sherman-Morrison-Woodbury
formula to the Drazin inverses of the generalized Schur complements $S=A-CD^{d}B$ and $Z=D-BA^{d}C$.

\begin{theorem}\label{ApiCDdB=0}
 If $A^{\pi}CD^{d}B=0$ and $KD^{\pi}Z^{d}H=KD^{d}Z^{\pi}H$, then
$$
  S^{d}=A^{d}+KZ^{d}H+\sum_{i=0}^{k-1}(A^{d}+KZ^{d}H)^{i+2}SA^{i}A^{\pi},
$$
where $H=BA^{d},~K=A^{d}C$;
or alternatively
\begin{multline*}
S^{d}=A^{d}+A^{d}CZ^{d}BA^{d}-\sum_{i=0}^{k-1}(A^{d}+A^{d}CZ^{d}BA^{d})^{i+1}A^{d}CZ^{d}BA^{i}A^{\pi}\\
 +\sum_{i=0}^{k-1}(A^{d}+A^{d}CZ^{d}BA^{d})^{i+1}A^{d}C(Z^{d}D^{\pi}-Z^{\pi}D^{d})BA^{i},
\end{multline*}
where
$k=\ind(A)$.
\end{theorem}
\begin{proof} Since $A^{\pi}CD^{d}B=0$, we have $S_A=AA^dSAA^d=SAA^d$. Thus the first equation follows from Lemma \ref{(lemma) P nilpotent} and Lemma \ref{(lemma) Qd}. The second equation follows from the first one and the fact that
\begin{align*}
(A^{d}+KZ^{d}H)SA^{\pi}=&(AA^{d}-KD^{d}B+KZ^{d}BA^{e}-KZ^{d}(D-Z)D^{d}B)A^{\pi}\\
                       =&K(Z^{d}D^{\pi}-Z^{\pi}D^{d})B-KZ^{d}BA^{\pi}.
\end{align*}
\end{proof}

Theorem \ref{ApiCDdB=0} generalizes \cite[Theorem 2.1]{Cvetkovic (2012)},  \cite[Theorem 1]{Dijana 2013},                \cite[Theorem 2]{Dijana 2013}, \cite[Theorem 2.1]{Dopazo (2013)}, \cite[Theorem 2.1]{Shakoor} and  \cite[Theorem 2.1]{Wei (2002)}.

\begin{corollary}
   If $A^{\pi}CD^{d}B=0$, $CD^{\pi}Z^{d}B=0$ and $CD^{d}Z^{\pi}B=0$, then
\begin{multline*}
S^{d}=A^{d}+A^{d}CZ^{d}BA^{d}-\sum_{i=0}^{k-1}(A^{d}+A^{d}CZ^{d}BA^{d})^{i+1}A^{d}CZ^{d}BA^{i}A^{\pi}\\
 +\sum_{i=0}^{k-1}(A^{d}+A^{d}CZ^{d}BA^{d})^{i+1}A^{d}C(Z^{d}D^{\pi}-Z^{\pi}D^{d})BA^{i},
\end{multline*}
where
$k=\ind(A)$.
\end{corollary}

\begin{corollary}
 If $A^{\pi}CD^{d}B=0$ and $D^{\pi}=Z^{\pi}$, then
$$
S^{d}=A^{d}+A^{d}CZ^{d}BA^{d}-\sum_{i=0}^{k-1}(A^{d}+A^{d}CZ^{d}BA^{d})^{i+1}A^{d}CZ^{d}BA^{i}A^{\pi},
$$
where
$k=\ind(A)$.
\end{corollary}

The following theorem can be proved similar to Theorem \ref{ApiCDdB=0}.
\begin{theorem}\label{CDdBApi=0}
 If $CD^{d}BA^{\pi}=0$ and $KZ^{\pi}D^{d}H=KZ^{d}D^{\pi}H$, then
$$
  S^{d}=A^{d}+KZ^{d}H+\sum_{i=0}^{k-1}A^{i}A^{\pi}S(A^{d}+KZ^{d}H)^{i+2},
$$
or alternatively
\begin{multline*}
  S^{d}=A^{d}+A^{d}CZ^{d}BA^{d}-\sum_{i=0}^{k-1}A^{\pi}A^{i}CZ^{d}BA^{d}(A^{d}+A^{d}CZ^{d}BA^{d})^{i+1}\\
  +\sum_{i=0}^{k-1}A^{i}C(D^{\pi}Z^{d}-D^{d}Z^{\pi})BA^{d}(A^{d}+A^{d}CZ^{d}BA^{d})^{i+1},
\end{multline*}
where
$k= \ind(A) $.
\end{theorem}

Theorem \ref{CDdBApi=0} generalizes \cite[Theorem 2.1]{Cvetkovic (2012)},   \cite[Theorem 3]{Dijana 2013}, \cite[Theorem~2.2]{Dopazo (2013)}, \cite[Theorem 2.2]{Shakoor} and \cite[Theorem 2.1]{Wei (2002)}.

The following lemma is an immediate corollary of \cite[Corollary 3.2]{Chen (2008)}.
\begin{lemma}\label{(lemma)Chen Xu Wei(2008)}
 Let $P,~Q$ and $R$ be $n\times n$ matrices such that $PQ=QP=QR=RP=R^{2}=0$. If $Q$ is nilpotent,
 then
 $$
   (P+Q+R)^{d}=P^{d}+\sum_{i=0}^{t-1}(P^{d})^{i+2}RQ^{i},
 $$
  where $t=\ind(Q)$.
\end{lemma}

Let $H=BA^{d},~K=A^{d}C$ and $\Gamma=HK$, which were introduced by Wei \cite{Wei (1998)Expressions}  to give representations for the Drazin inverses of block matrices.

\begin{lemma}\label{Sd=ESEd+X}
 If $A^{\pi}CD^{d}B=0$ and $K\Gamma^dHSA^d=0$,
 then
 $$
 S^d=(ESE)^{d}+\sum_{i=0}^{t-1}((ESE)^{d})^{i+2}(SE^{\pi})^{i+1},
$$
where  $E=AA^{d}-K\Gamma^{d}H$ and $t=\ind(E^\pi SE^\pi)$.
\end{lemma}
\begin{proof}
Since $A^{\pi}CD^{d}B=0$, we have $A^{\pi}SE=A^\pi AE=0$. Thus
    $$E^{\pi}SE=(A^{\pi}+K\Gamma^{d}H)SE=K\Gamma^{d}HSE=0$$ and
$$
 E^{\pi}SE^{\pi}=E^{\pi}S=A^{\pi}S+K\Gamma^{d}HS=A A^{\pi}+K\Gamma^{d}HS.
$$
Since $A^{\pi} K=0$, $(K\Gamma^{d}HS)^2=0$ and $AA^{\pi}$ is nilpotent, we see that $E^{\pi}SE^{\pi}$ is nilpotent.
Now we use Lemma \ref{(lemma)Chen Xu Wei(2008)} to
$  S=ESE+E^{\pi}SE^{\pi}+ESE^{\pi}$ to obtain
      \begin{align*}
       S^{d}&=(ESE)^{d}+\sum_{i=0}^{t-1}((ESE)^{d})^{i+2}SE^{\pi}(E^{\pi}SE^{\pi})^{i}\\
      &=(ESE)^{d}+\sum_{i=0}^{t-1}((ESE)^{d})^{i+2}(SE^{\pi})^{i+1},
      \end{align*}
where $t=\ind(E^{\pi}SE^{\pi})$.
\end{proof}

\begin{lemma}\label{ESEd=EAdE}
 If $K\Gamma^{d}HSA^d=0$ and $K\Gamma^{\pi}D^{d}H=0$, then $ESE$ has the group inverse  and
$  (ESE)^{\#}=EA^{d}E
$, where $E=AA^{d}-K\Gamma^{d}H$.
\end{lemma}
\begin{proof}
Note that
\begin{equation}\label{eq:EK}
  EK=K-K\Gamma^{d}\Gamma=K\Gamma^{\pi}
\end{equation}
and $E=A^d(A-C\Gamma^dH)$. Then
\begin{multline*}
      EA^{d}ESE=EA^{d}SE-EA^{d}K\Gamma^{d}HSE\\
               =E(A^{e}-KD^{d}B)E
               =E-EKD^{d}BE
               =E-K\Gamma^\pi D^{d}BE
               =E.
\end{multline*}
It follows from Lemma \ref{ECE} that  $ESEA^{d}E=E$. Thus $ESE$ has the group inverse $EA^{d}E$.
\end{proof}

\begin{theorem}\label{Sd=(EAdEd),ApiCDdH=0}
  If $A^{\pi}CD^{d}B=0,~K\Gamma^{d}HSA^{d}=0$ and $K\Gamma^{\pi}D^{d}H=0$, then
\begin{multline*}
  S^{d}=(I-K\Gamma^{d}H)A^{d}(I-K\Gamma^{d}H)-\sum_{i=0}^{k-1}((I-K\Gamma^{d}H)A^{d})^{i+2}K\Gamma^{d}HSA^{i}\\
  -\sum_{i=0}^{k-1}((I-K\Gamma^{d}H)A^{d})^{i+1}K\Gamma^{\pi}D^{d}BA^{i},
\end{multline*}
where $H=BA^{d},~K=A^{d}C,~\Gamma=HK$ and $k=\ind(A)$.
\end{theorem}
\begin{proof}
Let $E=AA^{d}-K\Gamma^{d}H$. By Lemma \ref{Sd=ESEd+X} and Lemma \ref{ESEd=EAdE} we have
\begin{align}
   S^{d}&=EA^{d}E+\sum_{i=0}^{t-1}(EA^{d}E)^{i+2}(SE^{\pi})^{i+1}\nonumber\\
&=EA^{d}E+\sum_{i=0}^{t-1}(EA^{d}E)^{i+1}(EA^{d}ESE^{\pi})(E^{\pi}SE^{\pi})^{i},\label{eq:Th2a}
\end{align}
where $t=\ind(E^{\pi}SE^\pi)$.
Note that $EA^{d}ESE^{\pi}=EA^{d}ESA^{\pi}+EA^{d}ESK\Gamma^{d}H$. Since $K\Gamma^{d}HSA^{\pi}=K\Gamma^{d}HS$, we have by \eqref{eq:EK}
\begin{align*}
EA^{d}ESA^{\pi}&=EA^{d}(AA^{d}SA^{\pi}-K\Gamma^{d}HSA^{\pi})\\
&=EA^{d}(-CD^{d}BA^{\pi}-K\Gamma^{d}HS)\\
&=-EKD^{d}BA^{\pi}-EA^{d}K\Gamma^{d}HS\\
&=-K\Gamma^{\pi}D^{d}B-EA^{d}K\Gamma^{d}HS.
\end{align*}
Since $K\Gamma^dHSK=0$, we have \begin{align*}
   EA^{d}ESK\Gamma^{d}H&=EA^{d}(A^{e}-K\Gamma^{d}H)SK\Gamma^{d}H\\
   &=EA^{d}SK\Gamma^{d}H\\
   &=EA^{d}(A-CD^dB)K\Gamma^{d}H\\
   &=EK\Gamma^{d}H-EKD^{d}HC\Gamma^{d}H\\
   &=K\Gamma^{\pi}\Gamma^{d}H-K\Gamma^{\pi}D^{d}HC\Gamma^{d}H\\
   &=0.
\end{align*}
Then
$EA^{d}ESE^{\pi}=-K\Gamma^{\pi}D^{d}B-EA^{d}K\Gamma^{d}HS$.
Note that $ E^{\pi}SE^{\pi}=A A^{\pi}+K\Gamma^{d}HS$.
Since $A^\pi K=0$ and $(K\Gamma^dHS)^2=0$, we have $$(E^{\pi}SE^{\pi})^{i}=(A^{\pi}A+K\Gamma^{d}HS)A^{i-1},$$
for any positive integer $i$, whence $\ind (A)\leq\ind(E^{\pi}SE^{\pi})\leq \ind(A)+1$.
It follows that
\begin{align}
  EA^{d}ESE^{\pi}(E^{\pi}SE^{\pi})^{i}=&(-K\Gamma^{\pi}D^{d}B-EA^{d}K\Gamma^{d}HS)(A^{\pi}A+K\Gamma^{d}HS)A^{i-1}\nonumber\\
            =&-K\Gamma^{\pi}D^{d}BA^{i}-EA^{d}K\Gamma^{d}HSA^{i},\label{eq:Th2b}
\end{align}
for any positive integer $i$.
Combining \eqref{eq:Th2a} and \eqref{eq:Th2b} yields
$$
  S^{d}=EA^{d}E-\sum_{i=0}^{t-1}(EA^{d})^{i+2}K\Gamma^{d}HSA^{i}-\sum_{i=0}^{t-1}(EA^{d})^{i+1}K\Gamma^{\pi}D^{d}BA^{i}.
$$
Since $t-1\leq k\leq t$ and $K\Gamma^{d}HSA^{i}=K\Gamma^{\pi}D^{d}BA^{i}=0$ for any $i\geq k$, we have
$$
  S^{d}=EA^{d}E-\sum_{i=0}^{k-1}(EA^{d})^{i+2}K\Gamma^{d}HSA^{i}-\sum_{i=0}^{k-1}(EA^{d})^{i+1}K\Gamma^{\pi}D^{d}BA^{i}.
$$
Now the conclusion follows from the fact that $EA^{d}=(I-K\Gamma^{d}H)A^{d}$ and $A^{d}E=A^{d}(I-K\Gamma^{d}H)$.
\end{proof}

Theorem \ref{Sd=(EAdEd),ApiCDdH=0} generalizes \cite[Theorem 2.2]{Cvetkovic (2012)}, \cite[Theorem 2.8]{Shakoor} and \cite[Theorem 2.2]{Wei (2002)}.

\begin{corollary}
 If $A^{\pi}CD^{d}B=0,~C\Gamma^{d}ZD^{d}B=0,~C\Gamma^{d}D^{\pi}B=0$ and $C\Gamma^{\pi}D^{d}B=0$, then
$$
  S^{d}=(I-K\Gamma^{d}H)A^{d}(I-K\Gamma^{d}H)+\sum_{i=0}^{k-1}((I-K\Gamma^{d}H)A^{d})^{i+2}K\Gamma^{d}BA^{i}A^{\pi},
$$
where $k=\ind(A)$.
\end{corollary}
\begin{proof}  Simple computations show that $K\Gamma^{d}HS=-K\Gamma^{d}BA^{\pi}$ and $K\Gamma^{\pi}D^{d}B=0$,
 from which Theorem \ref{Sd=(EAdEd),ApiCDdH=0} gives the desired result.
\end{proof}

The following theorem may be proved in the same way as Theorem \ref{Sd=(EAdEd),ApiCDdH=0}.
\begin{theorem}\label{(2)KDdBApi=0}
  If $CD^{d}BA^{\pi}=0,~KD^{d}Z\Gamma^{d}H=0,~KD^{\pi}\Gamma^{d}H=0$ and $KD^{d}\Gamma^{\pi}H=0$, then
\begin{multline*}
  S^{d}=(I-K\Gamma^{d}H)A^{d}(I-K\Gamma^{d}H)-\sum_{i=0}^{k-1}A^{i}SK\Gamma^{d}H(A^{d}(I-K\Gamma^{d}H))^{i+2}\\
    -\sum_{i=0}^{k-1}A^{i}CD^{d}\Gamma^{\pi}H(A^{d}(I-K\Gamma^{d}H))^{i+1},
\end{multline*}
where $H=BA^{d},~K=A^{d}C,~\Gamma=HK$ and $k=\ind(A)$.
\end{theorem}

 Theorem \ref{(2)KDdBApi=0} generalizes \cite[Theorem 2.2]{Cvetkovic (2012)}, \cite[Theorem 2.9]{Shakoor} and \cite[Theorem 2.2]{Wei (2002)}.

\section{Generalized Jacobson's Lemma}
\begin{lemma}\label{triangle}(\cite{Hartwig (1977)} and \cite{Meyer (1977)})
  Let $M=\left(\begin{array}{cc}
                   A & C \\
                   0 & D
                 \end{array}\right)$
     and $N=\left(\begin{array}{cc}
                   D & 0 \\
                   C & A
                 \end{array}\right)\in \mathbb{C}^{n\times n}$,
  where $A$ and $D$ are square matrices. Then
  $$M^{d}=\left(\begin{array}{cc}
                   A^{d} & X \\
                   0 & D^{d}
                 \end{array}\right) ~~~\text{and}~~
  N^{d}=\left(\begin{array}{cc}
                   D^{d} & 0 \\
                   X & A^{d}
                 \end{array}\right),$$
   where $$
   X=\sum_{i=0}^{s-1}(A^{d})^{i+2}CD^{i}D^{\pi}+A^{\pi}\sum_{i=0}^{r-1}A^{i}C(D^{d})^{i+2}-A^{d}CD^{d},$$
 $r=\ind(A)$ and $s=\ind(D)$.
\end{lemma}

\begin{lemma}\label{Jacobson,Schur,A^{e}SA^{e}}
If $D^{d}BAA^{d}=D^{d}DBA^{d}$, then
$$
   S_{A}^{d}=A^{d}+A^{d}CZ_{D}^{d}D^{d}BAA^{d}-\sum_{i=0}^{s-1}(A^{d})^{i+2}CDD^{d}Z_{D}^{i}Z_{D}^{\pi}D^{d}BAA^{d},
$$
where $S_{A}=AA^{d}SAA^{d},~Z_{D}=DD^{d}ZDD^{d}$ and $s=\ind(Z_{D})$.
\end{lemma}
\begin{proof}By abuse of notation we write $A^{2d}$ instead of $(A^d)^2$ for any $n\times n$ matrix $A$.
Let $A^{e}=AA^{d}$ and $D^{e}=DD^{d}$.
Note that
$$
    \left(\begin{array}{cc}
   S_{A} & A^{e}CD^{e} \\
    0    & DD^{e}
   \end{array}\right)   =\left(\begin{array}{cc}
                                AA^{e} & A^{e}CD^{e} \\
                                D^{e}BA^{e} & DD^{e}
                             \end{array}\right)
                             \left(\begin{array}{cc}
                               I & 0 \\
                               -D^{d}BA^{e} & I
                             \end{array}\right).
$$
For short let us introduce the temporary notation
$$
                                          M =\left(\begin{array}{cc}
                                AA^{e} & A^{e}CD^{e} \\
                                D^{e}BA^{e} & DD^{e}
                             \end{array}\right)~~~\text{and}~~~
                          N =  \left(\begin{array}{cc}
                               I & 0 \\
                               -D^{d}BA^{e} & I
                             \end{array}\right).
$$
Then Cline's formula gives
 \begin{align*}
                           &\left(\begin{array}{cc}
                             S_{A} & A^{e}CD^{e} \\
                              0           & DD^{e}
                             \end{array}\right)^{d}
                            = M(NM)^{2d}N.
    \end{align*}
A calculation yields  $NM=\left(\begin{array}{cc}
                              AA^{e}                      & A^{e}CD^{e} \\
                              D^{e}BA^{e}-D^{d}BA^{e}A    & DD^{e}-D^{d}BA^{e}CD^{e}
                             \end{array}\right)$.
Since $D^{d}BAA^{d}=D^{d}DBA^{d}$, we have $D^{e}BA^{e}-D^{d}BA^{e}A=0$ and $DD^{e}-D^{d}BA^{e}CD^{e}=Z_{D}$.
Thus $NM=\left(\begin{array}{cc}
       AA^{e}  & A^{e}CD^{e} \\
          0    & Z_D
    \end{array}\right)$.
By Lemma~\ref{triangle} we have
$$(NM)^{2d} =\left(\begin{array}{cc}
                              A^{d}   & X \\
                              0       &  Z_{D}^{d}
                             \end{array}\right)^{2}
                          =\left(\begin{array}{cc}
                              A^{2d}   &     A^{d}X+XZ_{D}^{d}\\
                               0       &    Z_{D}^{2d}
                             \end{array}\right),
$$
where $X =\sum_{i=0}^{s-1}(A^{d})^{i+2}CD^{e}Z_{D}^{i}Z_{D}^{\pi}-A^{d}CZ_{D}^{d}$ and $s=\ind(Z_D)$.
Note that $A^{d}X+XZ_{D}^{d}=A^{d}X-A^{d}CZ_{D}^{2d}$.
Then
$$
   \left(\begin{array}{cc}
     S_{A} & A^{e}CD^{e} \\
      0           & DD^{e}
    \end{array}\right)^{d}
 =\left(\begin{array}{cc}
   A^{d}-XD^{d}BA^{e}         &   X \\
   D^{e}BA^{2d}-YD^{d}BA^{e}  &   Y
   \end{array}\right),
$$
where $Y= D^{e}B(A^{d}X-A^{d}CZ_{D}^{2d})+DZ_{D}^{2d}$.
Hence
$$
  S_{A}^{d}=A^{d}+A^{d}CZ_{D}^{d}D^{d}BA^{e}-\sum_{i=0}^{s-1}(A^{d})^{i+2}CD^{e}Z_{D}^{i}Z_{D}^{\pi}D^{d}BA^{e},
$$
where
$s=\ind(Z_{D})$.
\end{proof}

Combining Lemma \ref{(lemma) P nilpotent} and Lemma \ref{Jacobson,Schur,A^{e}SA^{e}} gives the following result.
\begin{lemma}\label{(lemma)Jacobson,Schur,S}
  If $A^{\pi}CD^{d}B=0~{and}~D^{d}BAA^{d}=D^{d}DBA^{d}$, then
$$
  S^{d}=S_{A}^{d}+\sum_{i=0}^{k-1}(S_{A}^{d})^{i+2}SA^{i}A^{\pi},
$$
where
$$
   S_{A}^{d}=A^{d}+A^{d}CZ_{D}^{d}D^{d}BAA^{d}-\sum_{i=0}^{s-1}(A^{d})^{i+2}CDD^{d}Z_{D}^{i}Z_{D}^{\pi}D^{d}BAA^{d},
$$
$Z_{D}=DD^{d}ZDD^{d}$, $k = \ind(A)$ and $s=\ind(Z_{D})$.
\end{lemma}

To represent $S^{d}$ in terms of $Z^d$ we need an extra assumption.

\begin{theorem}\label{Jacobson,Schur,S}
If $A^{\pi}CD^{d}B=0,~D^{\pi}BA^{d}C=0$ and $D^{d}BAA^{d}=D^{d}DBA^{d}$, then
$$
  S^{d}=S_{A}^{d}+\sum_{i=0}^{k-1}(S_{A}^{d})^{i+2}SA^{i}A^{\pi},
$$
where $k = \ind(A)$,
$$
   S_{A}^{d}=A^{d}+A^{d}CZ^{d}D^{d}BAA^{d}-\sum_{i=0}^{s-1}(A^{d})^{i+2}CDD^{d}Z^{i}Z^{\pi}D^{d}BAA^{d},
$$
and $s=\ind(Z)$.
\end{theorem}
\begin{proof}
Let $D^{e}=DD^{d}$ and $Z_{D}=D^{e}ZD^{e}$.
  Using an analogous strategy as Lemma \ref{(lemma) P nilpotent} we get
$$
  Z^{d}=Z_{D}^{d}+\sum_{i=0}^{t-1}(Z_{D}^{d})^{i+2}ZD^{i}D^{\pi},
$$
where $t= \ind(D)$.
Note that $Z^{d}D^{e}=Z_{D}^{d},~D^{e}Z^{d}=Z^{d}$ and $Z_{D}=ZD^{e}$.
Then $Z_{D}^{i}=Z^{i}D^{e}$ and $Z_{D}^{\pi}D^{d}=(I-Z_{D}Z^{d})D^{d}=Z^{\pi}D^{d}$,
implying $$Z_{D}^{i}Z_{D}^{\pi}D^{d}=Z^{i}D^{e}(I-Z^{e})D^{d}=Z^{i}Z^{\pi}D^{d}.$$
Now the theorem follows from Lemma~\ref{(lemma)Jacobson,Schur,S}.
\end{proof}

\begin{corollary}
  Let $A$ and $D$ be invertible. If $DB=BA$, then
$$
 (A-CD^{-1}B)^{d}=A^{-1}+A^{-1}CZ^{d}D^{-1}B-\sum_{i=0}^{s-1}A^{-i-2}CZ^{i}Z^{\pi}D^{-1}B,
$$
where $s=\ind(Z)$.
\end{corollary}

If $A$ and $D$ are identity matrices in the corollary above, then we recover a generalization of Jacobson's Lemma (see \cite[Theorem 3.6]{Castro (2010)}) for the case of matrices.
\begin{corollary}Let $s=\ind(I-BC)$. Then
$$
  (I-CB)^{d}=I+C(I-BC)^{d}B-\sum_{i=0}^{s-1}C(I-BC)^{i}(I-BC)^{\pi}B.
$$
\end{corollary}

\bibliographystyle{amsplain}

\end{document}